\documentclass[a4paper,11pt]{article}

\usepackage[a4paper, left=3.2cm,top=3cm,right=3.2cm,bottom=3.8cm]{geometry}
\usepackage{concmath}
\usepackage[T1]{fontenc}

\usepackage{authblk}

\usepackage{graphicx}
\usepackage{mathtools}
\usepackage{psfrag}
\usepackage{amssymb}
\usepackage{amsmath}
\usepackage{authblk}
\usepackage{siunitx}
\usepackage{bm}

\usepackage[utf8]{inputenc}

\usepackage{framed}
\usepackage{amsthm}
\usepackage[hyphens]{url}
\usepackage{hyperref}\usepackage{todonotes}

\usepackage{enumitem}
\setitemize{label=\scriptsize{$\blacksquare$},topsep=0em}
\setenumerate{label=(\alph*), topsep=0em}

\newcommand{\dom}{\mathbb{D}}

\newcommand{\id}{\operatorname{Id}}
\newcommand{\Io}{\mathbf   I}
\newcommand{\Fo}{\mathbf   F}
\newcommand{\Ro}{\mathbf   R}
\newcommand{\Ao}{\mathbf   A}
\newcommand{\Bo}{\mathbf   B}
\newcommand{\Mo}{\mathbf   M}

\newcommand{\boldz}{\boldsymbol z}
\newcommand{\bxi}{\boldsymbol \xi}
\newcommand{\boldeta}{\boldsymbol \eta}
\newcommand{\boldkappa}{\boldsymbol \kappa}
\newcommand{\boldd}{\boldsymbol d}

\newcommand{\signal}{\boldsymbol x}
\newcommand{\data}{\boldsymbol y}
\newcommand{\ndata}{\boldsymbol{y}^{\delta}}

\newcommand{\sph}{\mathbb{S}}

\newcommand{\X}{\mathbb X}
\newcommand{\Y}{\mathbb Y}

\newcommand{\Po}{\mathbf   P}
\newcommand{\Uo}{\mathbf   U}
\newcommand{\Vo}{\mathbf   V}
\newcommand{\Wo}{\mathbf   W}

\DeclareMathOperator*{\argmin}{arg\,min}

\newcommand{\al}{\alpha}
\newcommand{\la}{\lambda}
\newcommand{\La}{\Lambda}

\newcommand{\Soft}{\mathbb{S}}
\newcommand{\soft}{\operatorname{soft}}
\newcommand{\K}{\mathbb{K}}
\newcommand{\R}{\mathbb{R}}
\newcommand{\N}{\mathbb{N}}

\newcommand{\half}{\mathbb{H}}

\newcommand{\edot}{\, \cdot \, }

\newcommand{\Sign}{\operatorname{Sign}}
\newcommand{\sign}{\operatorname{sign}}
\newcommand{\range}{\operatorname{ran}}
\newcommand{\supp}{\operatorname{supp}}

\usepackage{soul}
\usepackage{xcolor}
\colorlet{lred}{red!40}
\colorlet{lgreen}{green!40}
\colorlet{lblue}{blue!40}

\newcommand\abs[1]{\left\vert#1\right\vert}
\newcommand\sabs[1]{{\lvert#1\rvert}}
\newcommand\norm[1]{{\left\Vert#1\right\Vert}}
\newcommand\snorm[1]{\Vert#1\Vert}
\newcommand{\enorm}{\left\|\;\cdot\;\right\|}
\newcommand\set[1]{{\left\{#1\right\}}}
\newcommand{\kl}[1]{\left(#1\right)}

\newcommand{\sset}[1]{\{#1\}}
\newcommand\inner[2]{\left\langle#1,#2\right\rangle}
\newcommand\sinner[2]{\langle#1,#2\rangle}

\def\plus{{\boldsymbol{\texttt{+}}}}

\newcommand\rmd[1]{\, \mathrm d #1}

\usepackage{enumitem}
\setitemize{label=\scriptsize{$\blacksquare$}}
\setenumerate{label=(\alph*)}

\newtheorem{theorem}{Theorem}
\newtheorem{lemma}[theorem]{Lemma}
\newtheorem{proposition}[theorem]{Proposition}

\newtheorem{remark}[theorem]{Remark}
\newtheorem{example}[theorem]{Example}

\newtheorem{definition}[theorem]{Definition}
\newtheorem{algorithm}[theorem]{Algorithm}

\numberwithin{equation}{section}
\numberwithin{theorem}{section}
\allowdisplaybreaks

\title{Sparse regularization of inverse problems  by operator-adapted frame thresholding}

\author{J\"urgen Frikel}
\affil{Department of Computer Science and Mathematics, OTH Regensburg\authorcr
Galgenbergstra{\ss}e 32, D-93053 Regensburg, Germany
\authorcr
{\tt  juergen.frikel@oth-regensburg.de}
}

\author{Markus~Haltmeier}
\affil{Department of Mathematics, University of Innsbruck\authorcr
Technikerstrasse 13, 6020 Innsbruck, Austria
\authorcr
{\tt  markus.haltmeier@uibk.ac.at}
}

\date{}

\begin{document}
\maketitle

\begin{abstract}
We analyze sparse frame based regularization of inverse problems by means of a diagonal frame decomposition (DFD) for the forward operator, which generalizes the SVD. The DFD allows to define a non-iterative (direct) operator-adapted frame thresholding approach which we show to provide a convergent regularization method with linear convergence  rates.  These results will be compared to the well-known analysis and synthesis variants of sparse
$\ell^1$-regularization which are usually implemented thorough iterative schemes.
If the frame is a basis (non-redundant case), the three versions of sparse regularization, namely synthesis and analysis variants of $\ell^1$-regularization as well as the DFD thresholding are equivalent.  However, in the  redundant case, those three approaches are pairwise  different.
\end{abstract}

\section{Introduction}
\label{sec:bio:introduction}

This paper is concerned with inverse problems of the form
\begin{equation}\label{eq:ip}
	\ndata  =  \Ao \signal + z  \,,
\end{equation}
where $\Ao \colon \dom(\Ao) \subseteq \X \to \Y$ is a linear operator between Hilbert spaces, and $z$ denotes the data distortion (noise).
We allow unbounded operators and assume that $\dom(\Ao)$ is dense. Moreover, we assume that the unknown object $\signal$ is an element of a closed subspace  space $ \X_0 \subseteq  \X$ on which $\Ao$ is bounded.
We are particularly interested in problems, where \eqref{eq:ip}
is ill-posed in which case the solution of
\eqref{eq:ip} (if existent) is either not unique or the solution operator is not continuous (hence, the solution process is unstable with respect to data perturbations). In order to stabilize the inversion of \eqref{eq:ip} one has to apply regularization methods, cf. \cite{EnglHankeNeubauer,SchGraGrosHalLen09}. The basic idea of regularization is to include a-priori information about the unknown object into the solution process.

In this paper, we use sparsity based regularization, where the a-priori assumption on the unknown object is  sparsity of $\signal$  with respect to a frame  $(u_\la)_{\la \in \La}$ of $\X$, cf. \cite{SchGraGrosHalLen09,grasmair2008sparse,daubechies2004iterative,lorenz2008convergence,ramlau2006tikhonov,burger2016complex}. That is, we regularize the recovery of $\signal$ from measurements \eqref{eq:ip} by enforcing sparsity of $\signal$ with respect to a suitably chosen frame of $\X$. Sparse regularization is well investigated and has been applied to many different imaging problems, and by now there are many algorithms available that implement sparse regularization. However, when dealing with frames, there are at least two fundamentally different concepts implementing sparsity, namely the synthesis and the analysis variant. The reason for this lies in the fact that expansions of $\signal\in \X$ with respect to frames are not unique (which is in contrast to basis expansions). In the synthesis variant, it is assumed that the unknown
is a sparse linear combination of frame elements, whereas, in the analysis variant, it is required that the inner products $\inner{u_\la}{\signal}$ with respect to a given frame are sparse.
The difference between these approaches has been pointed out clearly in \cite{elad2017analysis}.

Sparse regularization is widely used in inverse problems as it provides good regularization results and is able to preserve or emphasize features (e.g., edges) in the reconstruction. However, this often comes at the price of speed, since most of the algorithms implementing sparse regularization are based on variational formulations that are solved by iterative schemes.

In the present paper, we investigate a third variant of sparse regularization that is based on an operator-adapted diagonal frame decomposition (DFD) of the unknown object, cf. \cite{donoho19995nonlinear,candes2002recovering,colonna2010radon} and which allows to define a direct (non-iterative) sparse regularization method. In the noise-free case ($z=0$), explicit reproducing formulas for the unknown object can be derived from the DFD, where the frame coefficients of $\signal$ are calculated directly from the data $\data = \Ao\signal$. In the presence of noise ($z\neq 0$), regularized versions of those formulas are obtained by applying component-wise soft-thresholding to the calculated coefficients, where the soft-thresholding operator is defined as follows:

 \begin{definition}[Soft-thresholding]
 Let $\La$ be some index set.
\begin{itemize}
\item For
$\eta, d \in \K$ let  $\soft( \eta, d )
\coloneqq  \sign(\eta)
\max \set{0,  \abs{\eta} - d }$.

\item
For $\boldeta, \boldd \in \K^\Lambda$ we define
the  component-wise  soft-thresholding by
\begin{equation} \label{eq:soft2}
\Soft_{\boldd} (\boldeta)
\coloneqq
\kl{\soft \kl{ \eta_\la ,  d_\la }}_{\la \in \La } \,.
\end{equation}
\end{itemize}
\end{definition}

Here and below we  define  $\sign(\eta) \coloneqq  \eta / \abs{\eta}$ for $\eta \in  \K \setminus \{ 0 \}$ and $\sign(0) \coloneqq 0$.  The advantage of the DFD-variant of sparse regularization lies in the fact that it admits a non-iterative (direct) and fast implementation of sparse regularization that can be easily implemented for several inverse problems.

We point out, that the three variants of sparse regularization (mentioned above) are equivalent if orthonormal bases are used instead of frames, but they are fundamentally different in the redundant case.

As the main theoretical results in this paper we show that the third variant of sparse regularization, which we call
\emph{DFD-thresholding}, defines a convergent regularization method and we derive linear convergence rates for sparse solutions.
For the basis case, the same results follow from existing results of $\ell^1$-regularization \cite{daubechies2004iterative,grasmair2008sparse,grasmair2011necessary}. In the redundant case, the results follow from \cite{grasmair2011necessary} for the synthesis approach and from \cite{haltmeier2013stable} for the analysis approach. In case of DFD-thresholding, we are not aware of any results concerning convergence analysis or convergence rates.

\subsection*{Outline}

This paper is organized as follows.
In Section~\ref{sec:bfd} we define the
diagonal frame decompositions of operators and give several examples  of  diagonal frame
expansions   for  various operators
using wavelet, curvelet, and  shearlet
frames.       In Section~\ref{sec:ell1}
we review the convergence theory of
$\ell^1$-regularization and the convergence rates.
In Section~\ref{sec:bft}, we show that
DFD-thresholding is a convergent regularization method and derive its convergence rates.

\section{Diagonal frame decomposition}
\label{sec:bfd}

In this section, we introduce the concept of operator adapted diagonal frame decompositions (DFD) and discuss some classical examples of such DFDs in the case of the classical 2D Radon transform and the forward operator of photoacoustic tomography with a flat observation surface.

\subsection{Formal definition}

We define the operator adapted diagonal frame decomposition as a generalization of the wavelet vaguelette decomposition and the biorthogonal curvelet or shearlet decompositions to general frames, cf. \cite{donoho19995nonlinear,candes2002recovering,colonna2010radon}.

\begin{definition}[Diagonal frame decomposition (DFD)]
Let $(u_\la)_{\la \in \La} \in \X^\Lambda$, $(v_\la)_{\la \in \La} \in \Y^\Lambda$ and let $(\kappa_\la)_{\la \in \La}$ be a family of positive numbers. For a linear operator $\Ao \colon \dom(\Ao) \subseteq \X \to \Y$, we call $(u_\la, v_\la, \kappa_\la )_{\la \in \La}$ a \emph{diagonal frame decomposition (DFD) for $\Ao$}, if the following conditions hold:
\begin{enumerate}[label=(D\arabic*), leftmargin=3em]
\item  $(u_\la)_{\la}$ is a frame of $\X$,
\item \label{b2} $(v_\la)_{\la}$ is a  frame of $\overline{\range(\Ao)} = \overline{\Ao(\X)}$,
\item \label{b3} $\forall \la \in \La \colon \kappa_\la \neq 0 \wedge \Ao^* v_\la = \kappa_\la u_\la$.
\end{enumerate}
\end{definition}

\begin{remark}
The DFD  generalizes the  singular value decomposition (SVD) and the wavelet-vaguelette decomposition (WVD) (cf. \cite{donoho19995nonlinear}) as it allows the  systems $(u_\la)_\la$ and $(v_\la)_\la$ to be non-orthogonal and redundant.
Note that by \ref{b2} and \ref{b3}, the  frame $(u_\la)_{\la}$ satisfies $u_\la \in \Ao^* ( \overline{\range(\Ao)} ) = \range(\Ao^*) $, where we have made use of the identity  $ \range(\Ao)^\bot = \ker(\Ao^*)$.  For typical inverse problems  this yields a notable smoothness assumption on the involved elements of the frame  $(u_\la)_{\la}$.

Although the SVD has proven itself to be a useful tool for analyzing and solving inverse problems it has the following drawbacks: First, ONBs that are provided by the SVD (though optimally adapted to the operator in consideration), in many cases, don't provide sparse representations of signals of interest and, hence, are not suitable for the use in sparse regularization. In particular, frames that provide sparse representation of signals (such as wavelets or wavelet-like systems) are often not part of the SVD.  Second, SVD is often very hard to compute and not known analytically for many practical applications.

To overcome some of those difficulties, wavelet-vaguelette decompositions were introduced. Nevertheless, this concept builds upon expansions of signals with respect to orthogonal wavelet-systems, which may not provide an optimal sparse representation of signals of interest, e.g., signals with sharp edges. Thus, by allowing general frames, the BCD offers great flexibility in the choice of a suitable function system for sparse regularization while retaining the advantages.
\end{remark}

\begin{definition}
    For a frame $(u_\la)_{\la \in \La}$ of $\X$, the \emph{synthesis operator} is defined as
   $\Uo \colon \ell^2(\La)\to \Y \colon
        \bxi
        \mapsto
        \sum_{\la \in \La} \xi_\la u_\la$
       and the corresponding \emph{analysis operator} is defined as its adjoint, $
        \Uo^\ast \colon \X \to \ell^2(\La)\colon  \signal
        \mapsto  (\inner{\signal}{u_\la})_{\la\in\La}$.
        \end{definition}

 In what follows, the synthesis operator of a frame will be always denoted with the corresponding upper case letter, e.g., if $(v_\la)_{\la\in\La}$ is a frame, then $\Vo$ denotes the corresponding synthesis and $\Vo^\ast$ the analysis operator.

    In order to simplify the notation, we will also refer to  $(\Uo, \Vo, \boldkappa )$ the as DFD instead of using the full notation $(u_\la, v_\la, \kappa_\la )_{\la \in \La}$.

If a DFD exists for an operator $\Ao$, it immediately gives rise to a \emph{reproducing formula}
\begin{equation}
\label{eq:reproducing formula}
    \signal = \sum_{\la\in\La} \inner{\signal}{u_\la} \bar u_\la =  \sum_{\la\in\La} \kappa_\la^{-1} \inner{\Ao\signal}{v_\la} \bar u_\la = \bar\Uo\circ\Mo^\plus_{\boldkappa}\circ\Vo^\ast(\Ao\signal),
\end{equation}
where $(\bar u_\la)_{\la \in \La}$ is the dual frame to $(u_\la)_{\la \in \La}$ (cf. \cite{Christensen2003}) and $\bar\Uo$ the corresponding synthesis operator.
Moreover, $\Mo^\plus_{\boldkappa}$ denotes the Moore-Penrose inverse of $\Mo_{\boldkappa}$ and performs point-wise division with $\boldkappa$, i.e.
\begin{equation}
\label{eq:mpi}
(\Mo^\plus_{\boldkappa} (\boldeta))_ \la
\coloneqq
\begin{cases}
(\eta_\la/\kappa_\la)_{\la\in \La} & \text{ if }
\kappa_\la \neq 0 \\
0 & \text{ otherwise } \,.
\end{cases}
\end{equation}
Hence, from given (clean) data $\data$, one can calculate the frame coefficient of $\signal$ and obtain a reconstruction via \eqref{eq:reproducing formula}. The key to the practical use of this reproducing formulas is the efficient implementation of the analysis and synthesis operators $\Vo^\ast$ and $\bar\Uo$, respectively. For particular cases, we will provide efficient and easy to implement algorithms for the evaluation of $\Vo^\ast$ and $\bar\Uo$.

Note that, the reproducing formula \eqref{eq:reproducing formula} (similarly to the SVD) reveals the ill-posedness of the operator equation through the decay of the \emph{quasi-singular values} $\kappa_\la$. A regularized version of the reproducing formula \eqref{eq:reproducing formula} can be obtained by incorporating soft-thresholding of the frame coefficients. In section \ref{sec:bft}, we present a complete analysis of this approach as we are not aware of any results for the general DFD-thresholding in the context of regularization theory.

We now provide several examples of DFDs, including the wavelet vaguelette decomposition and the biorthogonal curvelet decomposition  \cite{candes2002recovering} for the Radon transform as well as a  DFD for the forward operator of photoacoustic tomography with flat observation surface.

\subsection{Radon transform}

\begin{definition}[Radon transform]
The  Radon transform $\Ro \colon L^2( B_1( 0 ) ) \to L^2( \sph ^1\times \R)$ is defined
 by
\begin{equation}
 \forall  (\theta, s) \in  \sph^1\times \R \colon  \quad
 \Ro f  (\theta, s) = \int_{\R} f(s \theta + t \theta^\bot) \rmd{s} \,.
\end{equation}
\end{definition}

It is well known that the  Radon transform  is bounded
on $L^2( B_1( 0 ))$, see  \cite{natterer01mathematics}.
Let $\Fo g (\theta, \omega ) = \int_{\R} g (\theta, s ) e^{-i \omega s}$ be the Fourier transform with respect to the first component and consider  the Riesz potential  \cite{natterer01mathematics}
\begin{equation} \label{eq:riesz}
(\Io^{-\alpha}  g) (\theta, \omega) \coloneqq \frac{1}{2\pi} \int_{\R} \abs{ \omega }^{\alpha}
(\Fo g) (\theta, \omega ) e^{i \omega s} \rmd{s}
\end{equation}
for  $\alpha > -1$.   The following hold:
\begin{enumerate}[label=(R\arabic*), leftmargin=3em]
\item \label{r1} The commutation relation $ (\Io^{-\alpha}   \circ  \Ro ) f =   (\Ro  \circ (-\Delta)^{ \alpha / 2 } )f $.
\item \label{r2} The filtered backprojection formula
$ f =  (4 \pi )^{-1} \Ro^*  (\Io^{-1}   \circ  \Ro ) f
=: \Ro^\sharp  \Ro f $.
\item \label{r3}
Isometry  property
$4 \pi \inner{f_1}{f_2}_{L^2} = \sinner{  \Io^{-1}   \circ  \Ro  f_1}{\Ro  f_2}_{L^2}$.
\end{enumerate}
Using these ingredients,  one can obtain a
DFD  for the  Radon transform as follows:

\begin{example}{DFD for the Radon transform}
Let $(u_\la)_{\la \in \La}$ be either an orthonormal basis of wavelets with compact support,  a (band-limited) curvelet or a shearlet tight frame with
$\la = (j, k, \beta) \in \Lambda $ where
$j \geq 0$ is the scale index.
Then $(\Uo, \Vo, \boldkappa)$ is a DFD  with
\begin{align} \label{eq:bcd-r1}
v_\la & \coloneqq
2^{-j/2} \,   (4\pi)^{-1}  (\Io^{-1}   \circ \Ro )u_\la
\\    \label{eq:bcd-r2}
\kappa_\la
& \coloneqq
2^{-j/2}
\,.
\end{align}
These results have been obtained in \cite{donoho19995nonlinear}
for wavelet bases, in  \cite{candes2002recovering} for curvelet systems and  in \cite{colonna2010radon} for the shearlet frame.
All cases are shown in similar manner and  basically follow
from \ref{r1}, \ref{r2} and the fact that
that  $  2^{-j/2} (-\Delta)^{1/4}  (4\pi)^{-1} u_\la \simeq  u_\la $ for any of the considered systems.
The  limited data case has been studied in \cite{frikel2013sparse}.
\end{example}

Equation \eqref{eq:bcd-r1} implies
\begin{equation}
\inner{g}{v_\la}
=
2^{-j/2} (4\pi)^{-1} \inner{g}{\Io^{-1}   \circ \Ro  u_\la}
=
2^{-j/2} (4\pi)^{-1} \inner{\Ro^*  \circ \Io^{-1}   g}{ u_\la} \,.
\end{equation}
This gives an efficient numerical algorithm for the evaluation of $\Vo^\ast$ provided that $\Uo^\ast$ is associated with an efficient
algorithm. This is in particular the case for the
wavelet, shearlet and  curvelet frames as above.

\begin{remark}
    We would like to note that in order to define a DFD for the case of curvelets or shearlets one needs to consider the Radon transform on subspaces of $L^2(\R^2)$ consisting of functions that are defined on unbounded domains (since band-limited curvelets or shearlets have non-compact support). However, since the Radon transform is an unbounded operator on $L^2(\R^2)$, the reproducing formula \eqref{eq:reproducing formula} will not hold any more in general. The reproducing formulas are at least available for the case that the object $\signal$ can be represented as a finite linear combination of curvelets or shearlets (cf. \cite{candes2002recovering} and \cite{colonna2010radon}).

    Another possibility would be to consider projections of curvelet or shearlet frames onto the space $L^2(B_1(0))$, which would yield a frame for this space (cf. \cite{Christensen2003}) and then define the DFD in the same way as above. Because the Radon transform is continuous on $L^2(B_1(0))$, the reproducing formula \eqref{eq:reproducing formula} will hold for general linear combinations.
    \end{remark}

\begin{framed}
\begin{algorithm}{Computing DFD coefficients
for  the Radon transform}
Let $\Uo$ be  a wavelet, shearlet of  curvelet  frame and define $\Vo$ by
 \eqref{eq:bcd-r1}.

\begin{enumerate}
\item Input: $g \in L^2(\sph^1 \times \R)$.
\item  Compute  $f_{\rm FBP} \coloneqq  (4\pi)^{-1} \Ro^* \circ   \Io^{-1} g $.
\item  Compute  $\boldeta  \coloneqq \Uo^* f_{\rm FBP}$
via wavelet, curvelet or shearlet transform.
\item  Apply rescaling $\boldeta  \gets (2^{-j/2}  \eta_\la)_{\la \in \La} $.
\item Output: Coefficients $\boldeta$.  \end{enumerate}
\end{algorithm}
\end{framed}

\subsection{Inversion of the wave equation}

We consider  a planar  geometry, which has been considered in  our previous work \cite{frikel2018efficient}.   Let $C_{0}^{\infty}(\half_+)$ denote  the space of compactly supported functions $h \colon \R^2  \to \R$ that are supported in the half space  $\half_+ \coloneqq\R \times(0,\infty)$. For $f \in C_{0}^{\infty}(\half_+)$ consider the initial value problem
\begin{equation} \label{eq:wave1}
\begin{aligned}
	(\partial_{t}^{2} - \Delta u(x,y,t) &= 0,
	&&(x,y,t)\in \R^2 \times \R\\
	u(x,y,0) &= y^{1/2} f (x,y)
	&&(x,y) \in\R^2\\
	\partial_{t}u(x,y,0) &= 0
	&& (x,y)\in\R^2 \,.
\end{aligned}
\end{equation}
The trace map  $\Ao \colon f \mapsto t^{-1/2} g $ where
$g(x,t) \coloneqq u(x,y=0, t) \chi \set{ t \geq 0} $
for $(x, t)  \in \R^2$    is known to be an isometry from $L^2(\half_+)$ to $L^2(\half_+)$,  see \cite{buhgeim78solution,frikel2018efficient,narayanan2010spherical}.
In particular, the operator $\Ao$ is continuous.

\begin{definition}[Forward operator for the wave equation]
We define $\Ao \colon  L^2(\half_+)\to L^2(\half_+)$ by
$ \Ao f \coloneqq t^{-1/2} u $, for  $f \in C_0^\infty (\half_+)$,
where $u$ is the solution of  \eqref{eq:wave1}, and extending it  by continuity
to $L^2(\half_+)$.
\end{definition}

The isometry  property implies that any frame
gives  a DFD  $(\Uo, \Vo, \boldkappa)$ by
setting     $v_\la = \Ao u_\la$ and $\kappa_\la =1$.
This, in particular, yields a wavelet vaguelette decomposition and a biorthogonal curvelet decomposition for the wave equation.

\begin{example}{DFD for the wave equation}
Let $(u_\la)_{\la \in \La}$ be either a  wavelet frame,  a curvelet frame or a shearlet frame with
$\la = (j, k, \beta) \in \Lambda $ where
$j \geq 0$ is the scale index.
Then $(\Uo, \Vo, \boldkappa)$ is a DFD  with
\begin{align} \label{eq:bcd-w1}
v_\la & \coloneqq
    \Ao u_\la
\\    \label{eq:bcd-w2}
\kappa_\la
& \coloneqq
1
\,.
\end{align}
As noted in \cite{frikel2018efficient} this  result directly follows from the isometry property     and the associated inversion formula
$f = \Ao^* \Ao f$.
\end{example}

The isometry property also gives an efficient
numerical algorithm for computing analysis coefficients
with respect to the frame   $\Vo$ in the case
that  $\Uo$ is associated with an efficient
algorithm.

\begin{framed}
\begin{algorithm}{Computing DFD coefficients
for  the wave equation}
Let $\Uo$ be the curvelet frame and define $\Vo$ by
 \eqref{eq:bcd-w1}.

\begin{enumerate}
\item Input: $g \in L^2(\half_+)$.
\item  Compute  $f_{\rm FBP} \coloneqq  \Ao^* g $.
\item  Compute  $\boldeta  \coloneqq \Uo^* f_{\rm FBP}$
via wavelet, curvelet or shearlet transform.
\item Output: Coefficients $\boldeta$.
\end{enumerate}
\end{algorithm}
\end{framed}

The algorithm described above can be used for any problem where the forward operator $\Ao$ is an isometry. In the case of the wavelet transform this simple procedure has been previously used in \cite{frikel2018efficient}.

\section{Sparse $\ell^1$-regularization}
\label{sec:ell1}

There are  two fundamentally different and well-studied instances of sparse frame based regularization, namely  $\ell^1$-analysis regularization and $\ell^1$-synthesis regularization. They are defined by
\begin{align}
\label{eq:anal}
\Bo_{\al}^{\rm ANA} (\data_\delta )
& \coloneqq
\argmin_{\signal\in \X} \set{
\frac{1}{2} \norm{\Ao \signal - \data_\delta}^2
+ \alpha \sum_{\la} d_\la \abs{\inner{u_\la}{\signal}}
}
\\
\label{eq:syn}
\Bo_{\al}^{\rm SYN} (\data_\delta )
& \coloneqq
\Wo \kl{ \argmin_{\bxi \in \ell^2(\Lambda)} \set{
\frac{1}{2} \norm{\Ao \Wo \bxi - \data_\delta}^2
+ \alpha \sum_{\la} d_\la \abs{\xi_\la} } }
 \,,
\end{align}
respectively, with weights $d_\la > 0$.

\begin{definition}
We call $\bxi \in \ell^2(\La)$ sparse  if the set
$\set{\la \in \La \mid \xi_\la \neq 0  }$ is finite.
\end{definition}

If $\bxi =(\xi_\la)_{\la \in \La}\in \ell^2(\La)$ is sparse,
we write $\Sign (\bxi) \coloneqq \{ \boldz =(z_\la)_{\la \in \La} \in \ell^2 (\La)
\mid z_\la \in \Sign (\xi_\la) \}$
where $\Sign (\, \cdot \,) \colon \R \to \R$ is the multi-valued signum function defined by $\Sign(0) = [-1,1]$ and
$\Sign(x)= \set{x/\abs{x}}$ for $x \neq 0$.
We will use the notation
\begin{equation} \label{eq:norm}
   \enorm_{\boldd, 1}
   \colon \ell^2(\La) \to \R \cup \set{\infty}
   \colon \bxi
   \mapsto
   \begin{cases}
   \sum_{\la} d_\la \abs{\xi_\la}
   & \text{ if }   (d_\la \xi_\la)_{\la \in \La} \in \ell^1(\Lambda)
   \\
   \infty
   & \text{ otherwise }\,.
\end{cases}
\end{equation}
Any  element in the set $\argmin \sset{\norm{\Uo^*(\signal)}_{\boldd, 1} \mid  \Ao \signal = \data}$ is called  $\norm{\Uo^*(\edot)}_{\boldd, 1} $-minimizing solution of $\Ao (\signal) = \data$.
Note that $\norm{\Uo^*(\edot)}_{\boldd, 1} $-minimizing solutions exists whenever
there is any solution $\signal $ with $\norm{\Uo^*(\signal)}_{\boldd, 1} < \infty$, as follows from \cite[Theorem 3.25]{SchGraGrosHalLen09}.

Below we recall  well-posedness and  convergence results  for both variants.
These results hold under the following  quite week assumptions:

\begin{framed}
\begin{enumerate}[label=(A\arabic*),leftmargin=3em]
\item \label{r-a1} $\Ao \colon \X \to \Y$ is bounded linear;
\item \label{r-a2} $\Uo, \Wo$ are synthesis operators  of frames $(u_\la)_{\la \in \La}$, $(w_\la)_{\la \in \La}$ of $\X$;
\item \label{r-a3}  $\boldd = (d_\la)_{\la \in \La} \in \R^\Lambda $ satisfies
$\inf \{d_\lambda \mid
\lambda \in \Lambda \} >0$.
\end{enumerate}
\end{framed}

For certain  sparse elements we will state  linear error estimates which have been derived in \cite{grasmair2008sparse, haltmeier2013stable}.
See \cite{daubechies2004iterative,lorenz2008convergence,ramlau2006tikhonov,burger2016complex,SchGraGrosHalLen09} for some further works on sparse $\ell^1$-regularization, and
\cite{dicken1996wavelet,louis1998wavelets,rieder1997wavelet} for wavelet regularization methods.

\subsection{$\ell^1$-Analysis  regularization}

Let us define the $\ell^1$-analysis Tikhonov functional
by
\begin{equation} \label{eq:tik-anal}
   \mathcal{A}_{\al, \data}
   \colon \X \to \R \cup \set{\infty}
   \colon \signal  \mapsto
   \frac{1}{2} \norm{\Ao \signal - \data}^2
   + \alpha \norm{\Uo^* \signal}_{\boldd,1} \,.
\end{equation}
Then  we have $\Bo_{\al}^{\rm ANA} (\data) = \argmin \mathcal{A}_{\al, \data}$.

\begin{proposition}[Convergence of analysis  regularization]
Let \ref{r-a1}--\ref{r-a3} be  satisfied, suppose
$\data \in \Y$, $\al >0$,
$(\data^k)_{k \in \N}\in \Y^\N$ with $\data^k \to \data$,
and choose $\signal^k \in \argmin \mathcal{A}_{\al, \data^k}$.
\begin{itemize}
\item \textsc{Existence:}
The functional  $\mathcal{A}_{\al, \data}$ has at least one minimizer.

\item  \textsc{Stability:}
There exists a subsequence $(\signal^{k(\ell)} )_{\ell \in \N}$  of
 $(\signal^k)_{k \in \N}$ and a minimizer $\signal_\alpha \in \argmin \mathcal{A}_{\al, \data} $
such that $\snorm{ \signal^{k(\ell)}  - \signal_\alpha} \to 0$.
If the minimizer $\signal_\alpha$ of $\mathcal{A}_{\al, \data}$ is unique, then
$\snorm{ \signal^k  - \signal_\alpha} \to 0$.

\item  \textsc{Convergence:}
Assume $\data = \Ao \signal$ for $\signal  \in \X $ with $\snorm{\Uo^* \signal}_{\boldd,1} < \infty$ and suppose $\snorm{\data^k-\data} \leq \delta_k$ with $(\delta_k)_{k\in \N} \to 0$. Consider a parameter choice $(\alpha_k)_k \in (0,\infty)^\N$
such that $\lim_{k \to \infty} \al_k = \lim_{k \to \infty} \delta_k^2/\al_k =0$.
Then there is an $\norm{\Uo^*(\edot)}_{\boldd,1}$-minimizing solution
$\signal^\plus $ of $\Ao (\signal)=\data$ and a subsequence
$(\signal^{k(\ell)})_{\ell\in\N}$ with $\snorm{\signal^{k(\ell)} - \signal^\plus}\to 0$.
If the $\snorm{\Uo^*(\edot)}_{\boldd,1}$-minimizing solution is unique,
then $\snorm{\signal^k \to  \signal^\plus}\to 0$.
\end{itemize}\label{prop:well-anal}
\end{proposition}

\begin{proof}
See \cite[Propositions 5, 6 and 7]{grasmair2008sparse}.
\end{proof}

In order to derive convergence rates, one has to make additional
assumptions on the exact solution $\signal^\plus$ to be recovered. Besides the sparsity this requires a certain
interplay between $\signal^\plus$  and the forward operator $\Ao$.

\begin{framed}
\begin{enumerate}[label=(A\arabic*),leftmargin=3em]\setcounter{enumi}{3}
\item \label{r-a4} $\Uo^* \signal^\plus$ is sparse;
\item \label{r-a5}
$\exists \boldz \in \Sign(\Uo^* \signal^\plus)
\colon
\Uo \boldz \in  \range (\Ao^*)   $;
\item \label{r-a6}
$\exists t \in(0,1)\colon \Ao \text{ is injective on }
\operatorname{span} \set{u_\lambda \colon \sabs{\boldz_\la} > t}$.
\end{enumerate}
\end{framed}

Assumption  \ref{r-a5} is a so-called source condition and the main restrictive
assumption. It requires that  there  exists  an element
$  \boldz \in   \Sign(\Uo^* \signal)$  that satisfies the smoothness assumption
$ \Uo \boldz \in \range (\Ao^*) $.
Because $\boldz \in \ell^2(\La)$, the space  $\operatorname{span} \set{u_\lambda \colon \abs{z_\la} > t} $ is finite dimensional. Therefore,
Condition \ref{r-a6} requires injectivity on a certain finite
dimensional subspace.

\begin{proposition}[Convergence rates for analysis regularization]
Suppose \ref{r-a1}--\ref{r-a6} hold. Then, for a parameter choice
$\alpha = C \delta$,  there is a constant $c_\plus \in (0,  \infty)$
such that for all $\data_\delta \in \Y$ with
$\snorm{ \Ao \signal -\data_\delta } \leq \delta$ and every minimizer
$\signal^{\alpha,\delta} \in \argmin \mathcal{A}_{\al, \data_\delta}$
we have $\snorm{\signal^{\alpha,\delta} - \signal^\plus } \leq c_\plus \delta $. \label{prop:rates-anal}
\end{proposition}

\begin{proof}
See \cite[Theorem III.8]{haltmeier2013stable}.
\end{proof}

\subsection{Synthesis regularization}

Let use denote  the $\ell^1$-synthesis Tikhonov functional by
\begin{equation} \label{eq:tik-syn}
   \mathcal{S}_{\al, \data}
   \colon \ell^2(\La)
   \to \R \cup \set{\infty}
   \colon \signal  \mapsto
   \frac{1}{2} \norm{\Ao \Wo \bxi - \data}^2
+ \alpha \norm{ \bxi}_{\boldd,1} \,.
\end{equation}
Then it holds $\Bo_{\al}^{\rm SYN} (\data) = \Wo  (\argmin \mathcal{S}_{\al, \data})$.
Synthesis regularization  can be seen as analysis regularization
for the coefficient inverse problem $\Ao \Wo \bxi = \data$
and the analysis operator $\Uo^*  = \id$.
Using Proposition~\ref{prop:well-anal}  we therefore
 have the  following result.

\begin{proposition}[Convergence of synthesis regularization]
Let \ref{r-a1}-\ref{r-a3} be  satisfied, suppose
$\data \in \Y$, $\al >0$, $(\data^k)_{k \in \N}\in \Y^\N$ with $\data^k \to \data$ and take
$\bxi^k \in \argmin \mathcal{S}_{\al, \data^k}$.

\begin{itemize}
\item \textsc{Existence:}
The functional  $\mathcal{S}_{\al, \data}$ has at least one minimizer.

\item  \textsc{Stability:}
There exists a subsequence $(\bxi^{k(\ell)} )_{\ell \in \N}$
of $(\bxi^k )_{k\in \N}$ and $\bxi_\alpha \in \argmin \mathcal{S}_{\al, \data} $
such that $(\bxi^{k(\ell)} )_\ell \to \bxi_\alpha$.
If  the minimizer of $\mathcal{S}_{\al, \data}$ is unique, then
$\snorm{ \bxi^k  - \bxi_\alpha} \to 0$.

\item  \textsc{Convergence:}
Assume $\data = \Ao \Wo \bxi$ for $\bxi  \in \ell^2(\Lambda) $ with $\snorm{ \bxi}_{\boldd,1} < \infty$ and  $\snorm{\data^k-\data} \leq \delta_k$ with $(\delta_k)_{k\in \N} \to 0$. Consider a parameter choice
$(\alpha_k)_{k\in \N} \in (0,\infty)^\N$
with $\lim_{k \to \infty} \al_k = \lim_{k \to \infty} \delta_k^2/\al_k =0$.
Then there exist an $\norm{\edot}_{\boldd,1}$-minimizing solution
$\bxi^\plus $ of $(\Ao \Wo) (\bxi) = \data$
and a subsequence $(\bxi^{k(\ell)})_{\ell\in\N}$ with
$\snorm{\bxi^{k(\ell)} - \bxi^\plus}\to 0$.
If the $\snorm{\edot }_{\boldd,1}$-minimizing solution is unique,
then $\snorm{\bxi^k - \bxi^\plus}\to 0$.
\end{itemize} \label{prop:well-syn}
\end{proposition}

\begin{proof}
Follows from Proposition~\ref{prop:well-anal}
with $\Uo = \id$ and $\Ao \Wo$ in place of $\Ao$.
\end{proof}

We have linear convergence rates under the following additional
assumptions on the element  to be recovered.

\begin{framed}
\begin{enumerate}[label=(S\arabic*),leftmargin=3em]\setcounter{enumi}{3}
\item \label{r-a4s} $\signal^\plus = \Wo \bxi^\plus$ where
$\bxi^\plus \in \ell^2(\Lambda)$ is sparse;
\item \label{r-a5s}
$\exists \boldz \in \operatorname{Sign}  (\bxi^\plus) \colon
\boldz = \range(\Wo^* \Ao^*)   $;
\item \label{r-a6s}
$\exists t \in(0,1)\colon \Ao \Wo \text{ is injective on }
\operatorname{span} \set{e_\lambda \colon \abs{\boldz_\al} > t}$.
\end{enumerate}
\end{framed}

\begin{proposition}[Convergence rates for synthesis regularization]
Suppose that \ref{r-a1}--\ref{r-a3} and \ref{r-a4s}--\ref{r-a6s} hold.
Then, for a parameter choice $\alpha = C \delta$,
there is a constant $c_\plus \in (0,  \infty)$
such that for all $\data_\delta \in \Y$ with
$\snorm{ \Ao \signal -\data_\delta } \leq \delta$, every minimizer
$\bxi^{\alpha,\delta} \in \argmin \mathcal{S}_{\al, \data_\delta}$
we have
$\snorm{\bxi^{\alpha,\delta} - \bxi^\plus } \leq c_\plus \delta $.\label{prop:rates-syn}
\end{proposition}

\begin{proof}
Follows from Proposition \ref{prop:rates-anal}.
\end{proof}

Because $\Wo$ is bounded, the above convergence
results can be  transferred to convergence in the signal  space $\X$.
In particular, we have stability $\snorm{ \Wo (\bxi^{k(\ell)}) - \Wo (\bxi_\alpha)} \to 0 $ and convergence $\snorm{\Wo \bxi^{k(\ell)} - \Wo \bxi^\plus}\to 0$ under the assumptions made
in Proposition~\ref{prop:well-syn}, and  linear convergence rates
$\snorm{\Wo \bxi^{\alpha,\delta} - \signal^\plus } \leq \tilde c_\plus \delta $
under the assumptions made in Proposition~\ref{prop:rates-syn}.

\subsection{Sparse regularization using an SVD}

In the special  case that $\Uo$ is part of an SVD, then analysis and
synthesis regularization  are equivalent
and can be computed explicitly by soft-thresholding
of the expansion coefficients.

\begin{theorem}[Equivalence in the SVD case]
Let  $(\Uo, \Vo, \boldkappa)$ be an SVD for $\Ao$, let $\data_\delta \in \Y$ and
consider \eqref{eq:anal}, \eqref{eq:syn} with $\Wo =  \Uo$.
Then
\begin{equation*}
 \Bo_{\al}^{\rm ANA}(\data_\delta) = \Bo_{\al}^{\rm SYN}(\data_\delta)
= \sset{  (\Uo  \circ
\Mo_{\boldkappa}^\plus \circ \Soft_{\al\boldd/\boldkappa}
\circ \Vo^\ast )
  (\data_\delta ) }\,,
\end{equation*}
equals the soft-thresholding estimator in the SVD system.

\label{thm:basis}
\end{theorem}

\begin{proof}
Because  $\Uo$ is an orthonormal basis of $\X$, we have $\signal = \Uo \bxi \Leftrightarrow
\xi = \Uo^* \signal$ which implies that
$\Bo_{\al}^{\rm ANA}(\data_\delta) = \Bo_{\al}^{\rm SYN}(\data_\delta)$.
Now let $\signal_\al \in \Bo_{\al}^{\rm SYN}(\data_\delta)$ be any minimizer of
the $\ell^1$-analysis Tikhonov functional $ \mathcal{A}_{\al, \data}$.
Let $\Po_{\range(\Ao)^\bot}$ denote the orthogonal projection on $\range(\Ao)^\bot$.
We have
\begin{align*}
   \mathcal{A}_{\al, \data}(\signal)
   & =
  \frac{1}{2} \norm{\Ao \signal - \data}^2
   + \alpha \norm{\Uo^* \signal}_{\boldd,1}
    \\
    &=
    \norm{\Po_{\range(\Ao)^\bot}(\data)}^2 + \sum_{\la \in \La} \frac{1}{2} \abs{\inner{\Ao \signal - \data}{v_\la}}^2
     +\sum_{\la \in \La}
 \al d_\la \abs{\inner{\signal}{u_\la}}
\\
&=
\norm{\Po_{\range(\Ao)^\bot}(\data)}^2 + \sum_{\la \in \La} \frac{1}{2} \abs{\kappa_\la \inner{\signal}{u_\la}- \inner{\data}{v_\la}}^2
+ \al d_\la \abs{\inner{\signal}{u_\la}} \,.
\end{align*}
The latter sum is minimized by componentwise soft-thresholding.
This shows  $\signal_{\al, \delta} = (\Uo \circ \Soft_{\al\boldd/\boldkappa^2} \circ
\Mo_{\boldkappa}^\plus
\circ \Vo^\ast )  (\data_\delta )
= (\Uo \circ
\Mo_{\boldkappa}^\plus \circ \Soft_{\al\boldd/\boldkappa}
\circ \Vo^\ast )  (\data_\delta ) $ and
concludes the proof.
 \end{proof}

In the case that  $(\Uo, \Vo, \boldkappa)$   is a redundant
DFD expansion and not an SVD, then \eqref{eq:anal}, \eqref{eq:syn}, and the soft-thresholding estimator
\begin{equation} \label{eq:soft-onb}
(\Uo
\circ
\Mo_{\boldkappa}^\plus
\circ \Soft_{\al\boldd/\boldkappa}
\circ \Vo^\ast )
  (\data_\delta )
  \end{equation}
  are all non-equivalent. Further, in this case, \eqref{eq:anal} and \eqref{eq:syn}  have to be  computed by iterative minimization algorithms. This requires repeated application  of the forward and adjoint problem and therefore is time consuming. In the following section, we  study
DFD thresholding which is the analog of \eqref{eq:soft-onb}
for redundant systems. Despite the non-equivalence to $\ell^1$-regularization,  we are able to derive
the same type of convergence results and linear convergence rates as for the analysis and synthesis variants of $\ell^1$-regularization.

\section{Regularization via DFD thresholding}
\label{sec:bft}

Throughout this section we fix the following  assumptions:

\begin{framed}
\begin{enumerate}[label=(B\arabic*),leftmargin=3em]
\item \label{r-b1} $\Ao \colon \X \to \Y$ is bounded linear.
\item\label{r-b2} $(\Uo, \Vo, \boldkappa)$   is s
DFD for $\Ao$.
\item \label{r-b3}  $\boldd = (d_\la)_{\la \in \La} \in \R^\Lambda $ satisfies
$\inf \{d_\lambda \mid
\lambda \in \Lambda \} >0$.
\end{enumerate}
\end{framed}

In this section we show well-posedness, convergence and convergence rates
for DFD soft-thresholding.

\subsection{DFD soft-thresholding}

Any  DFD  gives an explicit inversion  formula
$\signal = (\bar \Uo \circ \Mo_{\boldkappa}^\plus  \circ \Vo^*)  (  \Ao \signal)$ where
$\Mo_{\boldkappa}^\plus$ is defined by \eqref{eq:mpi}.
For  ill-posed problems,  $\kappa_\la \to 0$
and therefore  the above reproducing  formula
is unstable  when applied  to noisy data  $\data_\delta$
instead of  $\Ao \signal$. Below we stabilize the inversion
by  including  the soft-thresholding  operation.

 \begin{definition}[DFD soft-thresholding]
Let $(\Uo, \Vo, \boldkappa )$ be a  DFD for $\Ao$. We define the nonlinear
DFD soft-thresholding estimator  by
\begin{equation}\label{eq:softest}
\Bo_\alpha^{\rm DFD} \colon \Y \to \X \colon\data \mapsto (\bar \Uo \circ \Mo_{\boldkappa}^\plus
\circ \Soft_{\al\boldd/\boldkappa}   \circ  \Vo^*   )(\data)
 \,.
\end{equation}
\end{definition}

If $(u_\la)_\la$, $(v_\la)_\la$ are ONBs,
then Theorem \ref{thm:basis} shows that  \eqref{eq:anal} and \eqref{eq:syn} are  equivalent to \eqref{eq:softest}.
In the  case of general frames, $\Bo_\alpha^{\rm ANA}$,
$\Bo_\alpha^{\rm SYN}$ and $\Bo_\alpha^{\rm DFD}$  are all different.

As the main result in this paper we show that DFD soft-thresholding yields  the same theoretical results as  $\ell^1$-regularization.
Assuming  efficient implementations for $\bar \Uo$ and   $\Vo^*$,
the DFD estimator has the advantage that it can be calculated non-iteratively
and is therefore much faster  than $\Bo_\alpha^{\rm SYN}$ and $\Bo_\alpha^{\rm DFD}$.

Consider the $\ell^1$-Tikhonov functional for the multiplication
operator $\Mo_{\boldkappa}$,
\begin{equation}\label{eq:tikM}
\mathcal{M}_{\al, \boldeta}
   \colon \ell^2(\La) \to \R \cup \set{\infty}
   \colon \bxi \mapsto
   \frac{1}{2} \norm{\Mo_{\boldkappa} \bxi - \boldeta}^2
 + \alpha \norm{ \bxi}_{\boldd,1} \,.
\end{equation}
The proof strategy used in this paper is based on the
following Lemma.

\begin{lemma}[$\ell^1$-minimization for multiplication operators]\mbox{}
\begin{enumerate}
\item \label{lem:id1}
$\forall \alpha \in \R_{>0}   \forall \boldeta \in  \ell^2(\La) \colon $ $\mathcal{M}_{\al, \boldeta}$ has a unique minimizer.

\item \label{lem:id2}
$\forall \alpha \in \R_{>0}   \forall \boldeta \in  \ell^2(\La) \colon $
$( \Mo_{\boldkappa}^\plus \circ \Soft_{\al\boldd/\boldkappa}  )(\boldeta)
=\argmin \mathcal{M}_{\al, \boldeta}$.

\item \label{lem:id3}
$\Mo_{\boldkappa}^\plus \circ \Soft_{\al\boldd/\boldkappa}
\colon \ell^2(\Lambda) \to \ell^2(\Lambda)$ is continuous.

\item \label{lem:id4} $\Vo^*  \Ao = \Mo_{\boldkappa} \Uo^*$.
\end{enumerate}\label{lem:id}
\end{lemma}

\begin{proof}
Because $(\id,\id,\boldkappa )$  is an SVD for $\Mo_{\boldkappa}$,
Items \ref{lem:id1}, \ref{lem:id2} follow from Theorem~\ref{thm:basis}, the
equivalence of  $\ell^1$-regularization and soft-thresholding in the SVD case.
Item~\ref{lem:id3} follows from Proposition~\ref{prop:well-syn}.
Moreover, the equality  $(\Vo^*  \Ao \signal)_\la = \inner{v_\la}{\Ao \signal}
=  \inner{\Ao^*  v_\la}{\signal} = \kappa_\la  \inner{u_\la}{\signal} =(\Mo_{\boldkappa} \Uo^* \signal)_\la$ shows Item \ref{lem:id4}.
\end{proof}

Note that the continuity  of $\Mo_{\boldkappa}^\plus \circ \Soft_{\al\boldd/\boldkappa}$  (see  Item~\ref{lem:id2} in the above lemma) is not obvious  as it is the  composition of the soft thresholding $\Soft_{\al\boldd/\boldkappa}$ with the  discontinuous  operator in $\Mo_{\boldkappa}^\plus$. The characterization in Item~\ref{lem:id2}  of $\Mo_{\boldkappa}^\plus \circ \Soft_{\al\boldd/\boldkappa}$ as minimizer of the $\ell^1$-Tikhonov functional $\mathcal{M}_{\al, \boldeta}$  and the existing stability results for $\ell^1$-Tikhonov regularization yields an elegant way  to obtain the continuity of $\Mo_{\boldkappa}^\plus \circ \Soft_{\al\boldd/\boldkappa}$. Verifying  the continuity directly  would  also be possible but seems to be a harder task.  A similar comment applies to the proof of Theorem \ref{thm:well} where we use the convergence of $\mathcal{M}_{\al, \boldeta}$   to show convergence  the DFD soft-thresholding estimator $\Bo_\alpha^{\rm DFD}$.

\subsection{Convergence analysis}

In this section, we show  that $(\Bo_\alpha^{\rm DFD})_{\alpha >0}$ is well-posed and convergent.

\begin{theorem}[Well-posedness and convergence]
Let \ref{r-b1}-\ref{r-b3} be  satisfied, suppose
$\data \in \Y$ and let $(\data^k)_{k \in \N}\in \Y^\N$ satisfy
$\data^k \to \data$. \label{prop:well-bfd}
\begin{enumerate}
\item \label{well1} \textsc{Existence:}
$\Bo_\alpha^{\rm DFD}\colon \Y \to \X$ is well-defined for all $\al >0$.

\item \label{well2} \textsc{Stability:}
$\Bo_\alpha^{\rm DFD}\colon \Y \to \X$ is continuous for all $\al >0$.

\item \label{well3} \textsc{Convergence:}
Assume $\data = \Ao \signal$ for some $\signal  \in \X $ with $\snorm{\Uo^* \signal}_{\boldd,1} < \infty$, suppose $\snorm{\data^k-\data} \leq \delta_k$ with $(\delta_k)_{k\in \N} \to 0$ and consider a parameter choice $(\alpha_k)_{k\in \N} \in (0,\infty)^\N$ with $\lim_{k \to \infty} \al_k = \lim_{k \to \infty} \delta_k^2/\al_k =0$.
Then $\snorm{\Bo_{\alpha_k}^{\rm DFD}(\signal^k) - \signal^\plus}\to 0$.
\end{enumerate}\label{thm:well}
\end{theorem}

\begin{proof} \mbox{}
\ref{well1}, \ref{well2}:
According to Lemma \ref{lem:id}, the mapping
 $\Mo_{\boldkappa}^\plus \circ \Soft_{\al\boldd/\boldkappa}
\colon \ell^2(\Lambda) \to \ell^2(\Lambda)$ is well-defined and continuous.
Moreover, by definition we have $\Bo_\alpha^{\rm DFD} = \bar \Uo \circ
( \Mo_{\boldkappa}^\plus \circ \Soft_{\al\boldd/\boldkappa}   ) \circ  \Vo^*$ which implies existence and stability of DFD thresholding.\\
\ref{well3}: We have
\begin{multline}\label{eq:proof:aux1}
\snorm{\Bo_{\alpha_k}^{\rm DFD}(\data^k) - \signal^\plus}
= \snorm{\bar \Uo \circ ( \Mo_{\boldkappa}^\plus \circ \Soft_{\al\boldd/\boldkappa}   ) \circ  \Vo^* (\data^k)- \bar \Uo \Uo^* \signal^\plus}
\\
\leq \norm{\Uo} \; \snorm{ (\Mo_{\boldkappa}^\plus \circ \Soft_{\al\boldd/\boldkappa}  )  (\Vo^* \data^k)- \Uo^* \signal^\plus} \,.
\end{multline}
Moreover, $\snorm{\Vo^* \data^k- \Mo_{\boldkappa} \Uo^* \signal^\plus}
= \snorm{\Vo^* \data^k - \Vo^*  \Ao \signal^\plus} \leq \norm{\Vo} \snorm{\data^k -\Ao \signal^\plus} \leq \norm{\Vo} \delta_k$. Therefore, Proposition \ref{prop:well-syn}  and the
equality $ \argmin \mathcal{M}_{\al, \Vo^* \data^k} = (  \Mo_{\boldkappa}^\plus \circ \Soft_{\al\boldd/\boldkappa}   )(\Vo^* \data^k)$ shown in Lemma~\ref{lem:id} imply
$\snorm{ (\Mo_{\boldkappa}^\plus \circ \Soft_{\al\boldd/\boldkappa}  )  (\Vo^* \data^k)- \Uo^* \signal^\plus} \to 0$ for $k \to \infty$.
Together with  \eqref{eq:proof:aux1} this yields \ref{well3}
and completes the proof.
\end{proof}

\subsection{Convergence rates}

Next we derive linear convergence rates  for sparse solutions.
Let use denote by  $\supp (\xi)  \coloneqq  \set{\la\in \La \colon  \xi_\la \neq 0}$ the support of $\xi \in \ell^2(\La)$.
To derive the convergence rates, we assume the following for the exact solution
$\signal^\plus$ to be recovered.

\begin{framed}
\begin{enumerate}[label=(B\arabic*),leftmargin=3em]\setcounter{enumi}{3}
\item \label{r-b4}
$ \Uo^* \signal^\plus $ is sparse.
\item \label{r-b5}
$\forall \la \in \supp(\Uo^* \signal^\plus ) \colon \kappa_\la \neq 0$.
\end{enumerate}
\end{framed}

Note that assumptions \ref{r-b4}, \ref{r-b5} imply   the source
condition
\begin{equation*}
    \boldz \in \partial \snorm{\edot}_{\boldd,1} \cap  \range(\Mo_{\boldkappa}^*) \neq \emptyset \,.
\end{equation*}
is satisfied for some element $\boldz  \in \ell^2(\La)$ that can be
chosen such that   $\abs{z_\la} < 1$ for  $\la \not\in  \supp (\Uo^* \signal^\plus)$.
Moreover, it follows that $\Mo_{\boldkappa}$ is injective on
  $  \operatorname{span} \set{e_\lambda \mid \abs{\boldz_\la} > t}$ with
$t \coloneqq  \max \set{ \abs{\kappa_\la}  \mid \la \not\in  \supp (\Uo^* \signal^\plus) }$.
Because $\Uo^* \signal^\plus \in \ell^2(\La)$, we have $t<1$.

\begin{theorem}[Convergence rates]
Suppose that \ref{r-b1}--\ref{r-b5} hold. Then, for the parameter choice
$\alpha =C \delta$ with $C \in (0,  \infty)$,
there is a constant $c_\plus \in (0,  \infty)$
such that for all $\data_\delta \in \Y$ with
$\snorm{ \Ao \signal -\data_\delta } \leq \delta$ we have
$\snorm{\Bo_\alpha^{\rm DFD}(\data_\delta) - \signal^\plus} \leq c_\plus \delta $. \label{prop:rates-bfd}
\end{theorem}

\begin{proof}
As in the proof of Theorem \ref{thm:well} one obtains
\begin{align} \label{eq:proof-rates1}
\snorm{\Bo_\alpha^{\rm DFD}(\data^k) - \signal^\plus}
&\leq \norm{\Uo} \; \snorm{ ( \Mo_{\boldkappa}^\plus \circ \Soft_{\al\boldd/\boldkappa}  )  (\Vo^* \data^k)- \Uo^* \signal^\plus}
\\ \label{eq:proof-rates2}
\snorm{\Vo^* \data^k- \Mo_\la \Uo^* \signal^\plus}
& \leq \norm{\Vo} \delta \,.
\end{align}
According to the considerations below \ref{r-b4}, \ref{r-b5} the
conditions \ref{r-a4s}--\ref{r-a6s} are satisfied for the operator
$\Mo_{\boldkappa}$  in place of $\Ao$ and with $\Wo =\id$.
The convergence rates result in Proposition \ref{prop:rates-anal}, estimate \eqref{eq:proof-rates2}, and the identity $ \argmin \mathcal{M}_{\al, \Vo^* \data} =
( \Mo_{\boldkappa}^\plus \circ \Soft_{\al\boldd/\boldkappa}  ) (\Vo^* \data)$ shown in Lemma~\ref{lem:id} imply
$\snorm{ ( \Mo_{\boldkappa}^\plus \circ \Soft_{\al\boldd/\boldkappa}  ) (\Vo^* \data)- \Uo^* \signal^\plus} \leq  c \delta $. Together  with \eqref{eq:proof-rates1} this implies
$\snorm{\Bo_\alpha^{\rm DFD}(\data_\delta) - \signal^\plus} \leq c \norm{\Uo} \delta$ and concludes the proof.
\end{proof}

\section{Conclusion}

To overcome the inherent  ill-posedness of inverse  problems,  regularization methods incorporate available prior information about the unknowns to be reconstructed. In this context, a useful prior is sparsity with respect to a certain frame. There are at least two different regularization strategies implementing sparsity with respect to a frame, namely $\ell^1$-analysis
regularization and $\ell^1$-synthesis regularization.
In this paper, we analyzed  DFD-thresholding as  a third variant
of sparse regularization. One advantage of DFD-thresholding compared to other  sparse regularization methods is its non-iterative nature leading to fast algorithms.
Besides having a DFD, actually computing the DFD soft-thresholding estimator \eqref{eq:softest} requires the dual frame  $\bar \Uo$. While in the general situation, the DFD and the dual frame  have to be computed numerically, we have shown that for many practical examples (see Section~\ref{sec:bfd}) they are known explicitly and efficient algorithms are available for its numerical evaluation.

The DFD-approach presented in this paper is well studied in the context of statistical estimating using certain multi-scale systems. However, its analysis in the context of regularization theory has not been given so far. In this paper we closed this gap and presented a complete convergence analysis of
DFD-thresholding as regularization method.

\section*{Acknowledgement}

The work of M.H  has been supported by the Austrian Science Fund (FWF),
project P 30747-N32.

\end{document}